\documentclass[11pt]{amsart}
\usepackage{color}
\usepackage{amssymb}
\theoremstyle{plain}
\newtheorem{Thm}{Theorem}

\newtheorem{Lem}[Thm]{Lemma}

\begin{document}

\title[Kazdan-Warner's  problem on surfaces]
{Two solutions to Kazdan-Warner's problem on surfaces}

\author{Li Ma}
\address{Li MA,  School of Mathematics and Physics\\
  University of Science and Technology Beijing \\
  30 Xueyuan Road, Haidian District
  Beijing, 100083\\
  P.R. China }
\address{ Department of Mathematics \\
Henan Normal university, Xinxiang, 453007 \\
China}
 \email{lma17@ustb.edu.cn}

\thanks{Li Ma's research is partially supported by the National Natural
  Science Foundation of China (No.11771124)}

\begin{abstract}
In this paper, we study the sign-changing Kazdan-Warner's problem on two dimensional closed Riemannian manifold with negative Euler number $\chi(M)<0$. We show that once, the direct method  on convex sets is used to find a minimizer of the corresponding functional, then there is another solution via a use of the variational method of mountain pass. In conclusion, we show that there are at least two solutions to the Kazdan-Warner's problem on two dimensional Kazdan-Warner equation provided the prescribed function changes signs and with this average negative.

{ \textbf{Mathematics Subject Classification 2010}: 53C20, 35J60, 58G99.}

{ \textbf{Keywords}: Kazdan-Warner problem, mountain pass,
direct method on convex sets, multiple solutions.}
\end{abstract}

\maketitle

\section{Introduction}\label{sect0}

The aim of this paper is to study the sign-changing Kazdan-Warner's problem on two dimensional closed Riemannian manifold with $\chi(M)<0$ and we show that there are at least two solutions to \eqref{KW}. This non-uniqueness problem is open since 1974 and the precise result is stated below.

  For a given smooth function $K$ on a compact Riemannian manifold $(M,g)$, Kazdan-Warner \cite{KW} studied the following problem
  \begin{equation}\label{KW}
  -\Delta u+\alpha=K(x) e^{2u}, \ \ in  \ M,
  \end{equation}
  where $\Delta$ is the Laplacian operator of the metric $g$ (and which is $\sum_i\partial_i^2$ on $R^2$) and $\alpha$ is a given real constant. Integrating by part we have that necessarily,
$$
\alpha |M|=\int_M Ke^{2u}dv
$$
where
$|M|=\int_M dv$
 is the volume of $(M,g)$ with $dv$ being the volume element of the metric $g$. Assume that $\alpha<0$. Using the method of super and sub solutions, Kazdan-Warner have proved that there exists a constant $\alpha_0<0$ such that for each $\alpha\in (\alpha_0,0)$ there is a solution to \eqref{KW}.
 The critical number $\alpha_0$ is defined by
 $$
 \alpha_0=\inf \{\alpha; \text{\eqref{KW} ~ is~solvable ~for }~ \alpha\}.
 $$
 Using a very beautiful argument, Kazdan and Warner \cite{KW} have showed that $\alpha_0=\infty$ if and only if $K$ is nontrivial non-positive function on $M$.

For the case when the given smooth function $K$ is positive somewhere with negative average, by the result of Kazdan and Warner mentioned above we have the critical number $\alpha_0>-\infty$. W.Chen and C.Li \cite{CL} have proved that there is a solution $u_0$ to \eqref{KW} at $\alpha=\alpha_0$ and $u_0$ is the $H^1$ limit of stable solution sequence $\{u_k\}$ (corresponding to $\alpha_k\to \alpha_0+$ ). Their argument is based on the solutions obtained by Kazdan-Warner and they have used the variational argument to get their solution. Note that in this case, there is a a positive smooth function $\psi$ such that
$$
-\Delta \psi -2K(x) e^{2u_0}\psi=0, \ \ in  \ M.
$$
The remaining question is whether the solution to \eqref{KW} is unique.

We show that there is no uniqueness of solutions to \eqref{KW} when $\alpha\in (\alpha_0,0)$. We denote by $\chi(M)$ the Euler characteristic of for the surface $M$.
We have two solutions result to \ref{KW} in this case.
\begin{Thm}\label{mali2} On the 2-dimensional compact Riemannian manifold $(M,g)$ with $\chi(M)<0$,
for $\alpha\in (\alpha_0,0)$,
there are at least two solutions to \eqref{KW} provided the given smooth function $K$ is positive somewhere with negative average.
 \end{Thm}
 This result will be proved by the mountain pass argument below.
There are some interesting works about prescribed sign-changing Gauss curvature on closed surfaces and related works on scalar curvature problems. We mention the interesting works such as the classical books \cite{A} and \cite{He}, the papers of M.Berger \cite{B}, Chang-Yang \cite{CY}, Kazdan-Warner \cite{KW1},  Chen-Li \cite{CL}, \cite{DL}, Borer-Luca-Struwe \cite{St}, Ma-Hong \cite{MH}, etc, and one may refer to Ma \cite{M} \cite{MW1} and Ma-Wei \cite{MW} for more related references.

The plan of this paper is below. In section \ref{sect1} we consider the method of super and sub solutions (the monotone method) to obtain a solution to \eqref{KW} in any dimensions. In section \ref{sect2} we consider the mountain pass solution to \eqref{KW} in dimension two. The key step is the verification of Palais-Smale condition to the related functional. In the last section we give a variational argument of Kazdan-Warner's result.

\section{the first solution via the direct method on convex sets}\label{sect1}
 In this section, we use the direct method on convex sets to get a solution to \eqref{KW}, our construction is slightly different from the one used in \eqref{KW} on a compact Riemannian manifold $(M,g)$. We have the following result and the new part in it is the local minima property, which will play a role in the mountain pass argument in next section.

 \begin{Thm} We assume that on a compact Riemannian manifold $(M,g)$ of any dimension,
the smooth function $K$ changes signs and $
 \bar{K}<0$.
Then
 for any $\alpha\in (\alpha_0,0)$, there is a solution to \eqref{KW}, which is a local minimizer of the functional
 $$
I(u)=\int_M(|\nabla u|^2+\alpha u)dv-\int_M K e^{2u}dv
$$
on $H^1(M)$.
 \end{Thm}
The proof may be outlined below. Let $u_+=u_0$ be the solution obtained by Kazdan-Warner \cite{KW} for some $\alpha_1$, $\alpha_0<\alpha_1<\alpha$.
Then for any $\alpha\in (\alpha_1,0)$,
$$
 -\Delta u_++\alpha-K(x) e^{2u_+}> -\Delta u_++\alpha_1-K(x) e^{2u_+}=0,
$$
 i.e., $u_+$ is the super solution to \eqref{KW} for $\alpha\in (\alpha_1,0)$.

 To get a solution by the method of super and sub solutions, we need to find a sub solution $u_-<u_+$ to (\ref{KW}). We do this below.

Recall that $\alpha<0$. Note that for any real number $c$ very negative and $v_c=c$, we have
$$
\Delta v_c-\alpha+Ke^{v_c}=-\alpha+Ke^c>0, \ \ in \ M.
$$
Then $u_-=c (<u_+)$ is a sub solution to \eqref{KW}. Then we have a solution $u_1$ to \eqref{KW} in the interval $[c,u_+]$, which is a local minimizer of the functional
$$
I(u)=\int_M(|\nabla u|^2+\alpha u)dv-\int_M K e^{2u}dv
$$
on $H^1(M)$. We refer to \cite{C} for related references.

\section{the second solution: the mountain pass}\label{sect2}

In this section we consider the equation \eqref{KW} on the closed surface $(M,g)$ with $\chi (M)=\frac{1}{2\pi} \int_M kdv$, where $k$ is the Gauss curvature of $g$.  Note that Theorem \ref{thm2} below implies Theorem \ref{mali2}.

 \begin{Thm}\label{thm2} On the 2-dimensional compact Riemannian manifold $(M,g)$ with $\chi(M)<0$, there is a mountain-pass solution to \eqref{KW} provided
the given smooth function $K$ such that $K>0$ somewhere and $
 \bar{K}<0$.
 \end{Thm}
We now give the idea to prove this result. We recall the functional $$
I(u)=\int_M(|\nabla u|^2+\alpha u)dv-\int_M K e^{2u}dv
$$
on $H^1(M)$. Note that the solution $u_1$ obtained by the method of super and sub solutions can also be described as the minimizer of the functional $I$ on
$[u_-,u_+]\bigcap H^1(M)$. Hence, for any $\phi\in H^1(M)$, $<I''(u_1)\phi, \phi>\geq 0$, i.e.,
$$
\int_M(|\nabla \phi|^2-2K e^{2u}\phi^2)dv\geq 0.
$$
Recall that $u_1$ satisfies
$$
-\Delta u_1+\alpha=Ke^{2u_1}, \ \ in \ M,
$$
and then
$$
\int_MKe^{2u_1}dv=\alpha |M|.
$$
We give a remark here. If $u_1$ is a strict minimizer, then we have some uniform constant $c>0$ such that
$$
\int_M(|\nabla \phi|^2-2K e^{2u_1}\phi^2)dv\geq c\int_M\phi^2dv, \ \ \forall \ \phi\in H^1(M).
$$
Otherwise, we have a positive solution $\phi$ to the linear equation
$$
-\Delta\phi-2K e^{2u_1}\phi=0,\ \ in \ M.
$$
We want to find another solution of the form $u_1+u$ such that $u\not=0$ satisfies
$$
-\Delta (u_1+u)+\alpha=Ke^{2u_1+2u}, \ \ in \ M.
$$
Then we have hat
$$
-\Delta u=Ke^{2u_1}(e^{2u}-1), \ \ in \ M.
$$
So we look for a non-trivial critical point of the functional
$$
J(u)=\int_M|\nabla u|^2dv+\int_M R(x)(2u-e^{2u})dv
$$
on $H^1(M)$, where we have set $R(x)=Ke^{2u_1}$. Note that $\int_M Ke^{2u_1}=\alpha |M|$.
Recall that $u=0$ is a local minimizer of $J$ on $H^1(M)$.

For small $\epsilon>0$, we let $M_\epsilon=\{R(x)\geq \epsilon\}$ and $M_{-\epsilon}=\{-R(x)\geq \epsilon\}$. As above, we denote by $\bar{u}=\frac{1}{vol(M)}\int_Mu dv$.
We now choose a smooth function $w_0\in C_0^1(M_{\epsilon}\bigcup M_{-\epsilon})$ which is positive in $\{R(x)\geq \epsilon\}$ for some $\epsilon>0$ and $\bar{w}_0=0$. Note that
as $t\to \infty$,
 $$
 J(tw_0)=t^2\int_M |\nabla w_0|^2dv +\int_M R(x)(2tw_0- e^{2tw_0})dv\to -\infty
 $$
 and we may choose $t_0>0$ large such that $J(t_0w_0)< J(0)=0$.

Define
$$
X= H^1(M)
$$
and
 $$
 c=\inf_{\gamma\in \Gamma} \sup_{t\in [0,1]} J(\gamma(t)),
 $$
 where
 $$
 \Gamma=\{\gamma\in C([0,1],X); \gamma(0)=0, \gamma(1)=t_0w_0\}.
 $$
 Note that $c\geq 0$. We shall verify that $J$ satisfies the Palais-Smale condition on $X$. Then, there is a "mountain pass" critical point of $J$ on $X$ satisfies
$$
-\Delta u-R(x)e^{2u}+R(x)=0,
$$
weakly in $H^1(M)$.
By the result of K.C.Chang \cite{C}, $c$ is a mountain pass critical value of the function $J$ and may be obtained by the function $w\not=0$, which gives a solution, which is different from the solution $u_1$ obtained by the monotone method above. Thus as we have noted before, we get the proof of Theorem \ref{thm2}.

The main topic now is to check the Palais-Smale condition for the functional $J$ on $H^1(M)$.
Set $M_-=\{x\in M; R(x)<0\}$, which is a non-empty open set in $M$. We have the following compactness result for the functional $J$ on $H^1(M)$.

 \begin{Lem}\label{PS} The functional $J$ satisfies the Palais-Smale condition at the level $c\geq 0$ in the function space $X=H^1(M) $. That is, if any sequence
$\{u_k\}\subset X$ satisfies $J(u_k)\to c$ and $J'(u_k)\to 0$ in the dual space $X^*$, then there is a subsequence of $\{u_k\}$ converges in $X$.
 \end{Lem}
 \begin{proof}
 Assume that the sequence $\{u_k\}\subset X$ satisfies $J(u_k)\to c$ and $J'(u_k)\to 0$ in the dual space $X^*$. That is,
 \begin{equation}\label{PS1}
 \int_M (|\nabla u_k|^2+R(x)(2u_k-e^{2u_k}))dv\to c
 \end{equation}
and
\begin{equation}\label{PS2}
\int_M(\nabla u_k\cdot \nabla \phi +R(x)\phi-R(x)e^{2u_k}\phi )dv=\circ(\|\phi\|), \ \ \phi\in X
\end{equation}
where $\|\cdot||$ is the norm on $X$.
Set $\phi=1$ in \eqref{PS2}, we have
\begin{equation}\label{key}
\int_MR(x)(1-e^{2u_k})dv\to 0, \ i.e., \ \int_MR(x)e^{2u_k}dv\to \int_MR(x)dv=d<0.
\end{equation}
By \eqref{PS1}, we have
\begin{equation}\label{PS3}
 \int_M (|\nabla u_k|^2+R(x)(2u_k))dv\to c+d.
\end{equation}
Let
$$
\bar{u}_k=\frac{1}{vol(M)}\int_Mu_kdv:=c_k.
$$
Then we have
$$
 \int_M (|\nabla u_k|^2+2R(x)(u_k-c_k))dv+2c_kd\to c+d.
$$
Note that
\begin{equation}\label{KW2}
|\int_MR(x)(u_k-c_k))dv|\leq |R|_\infty \int_M|u_k-c_k|dv\leq \frac12  \int_M |\nabla u_k|^2+C_1
\end{equation}
for some $C_1$. Then
\begin{equation}\label{KW3}
 \int_M (|\nabla u_k|^2+2R(x)(u_k-c_k))dv+2c_kd\geq -C_1+2c_kd.
\end{equation}
Let $u_k^+(x)=\sup(u_k(x),0)$. We want to show that the sequence $\{u_k^+\}$ is locally bounded in $H^1_{loc}(M_-)$. Take any non-empty domain $D$ in $M_-$ with $dist(D,\partial M_-):=d>0$ and $-R(x)\geq \delta>0$ on $D$. We show that there is a constant $C(D)$ such that $\| u_k^+\|\leq d(D)$. Take any $p\in M_-$ such that $B_d:=B_d(p)\subset M_-$, where $B_r(p)$ is the geodesic ball centered at $p$ with radius $r>0$. Choose the cut-off function $\eta\in C_0^1(B_{d/2})$ such that $\eta=1$ on $B_{d/4}$ and $|\nabla \eta|^2/\eta\leq C$ for some uniform constant $C>0$. Set $\phi=u_k^+\eta^2$ in \eqref{PS2}. Then we have some uniform $C>0$ such that
$$
\int_M(\nabla u_k\cdot \nabla (u_k^+\eta^2)+R(x)u_k^+\eta^2 -R(x)e^{2u_k}u_k^+\eta^2 )dv\leq C(\|u_k^+\eta^2\|).
$$
Since
$$
\int_M\nabla u_k\cdot \nabla (u_k^+\eta^2) =\int_M|\nabla (\eta u_k^+)|^2-\int_M (u_k^+)^2|\nabla \eta|^2,
$$
we then have 
$$
\int_M(|\nabla (\eta u_k^+)|^2-R(x)e^{2u_k}u_k^+\eta^2 )dv\leq C\int_M |\nabla \eta|^2 (u_k^+)^2-\int_MR(x)u_k^+\eta^2+C(\|u_k^+\eta\|).
$$
By
$e^{2t}\geq t^3$ for any real $t$ and $|\nabla \eta|^2\leq C\eta$, we have,
$$
\int_M(|\nabla (\eta u_k^+)|^2+\delta \eta^2(u_k^+)^4)dv\leq C\int_M \eta (u_k^+)^2-\int_MR(x)u_k^+\eta^2+C(\|u_k^+\eta\|).
$$

Using the Holder inequality we then have
$$
\int_M(|\nabla (\eta u_k^+)|^2+\frac12\delta \eta^2(u_k^+)^4)dv\leq C(\delta)+C(\|u_k^+\eta\|),
$$
which implies that there is a uniform constant $C:=C(d)$ such that
$$
\int_M(|\nabla (\eta u_k^+)|^2+\frac14\delta \eta^2(u_k^+)^4)dv\leq C.
$$
We now show that $\int_Mu_k^2$ is uniformly bounded (and is equivalent to $c_k$ being a bounded sequence).  If this is true, then the Palais-Smale sequence is bounded in $X$. Then we may assume that there is a subsequence, still denoted by $u_k$, which weakly converges to $u$ in $X$. Then the subsequence converges in $X$ by \eqref{PS2} and the fact that for any $p>1$,
  $$
  e^{2u_k}\to  e^{2u}, \ \ in \ L^p(M).
  $$

 To show $|u_k|_2^2:=\int_Mu_k^2$ being uniformly bounded, we argue by contradiction and assume $|c_k|\to\infty$. Let
 $$v_k=u_k/|u_k|_2=(w_k+c_k)/\sqrt{|w_k|_2^2+c_k^2Vol(M)}.$$
  Then by \eqref{PS3} we have
 $$
 \int_M |\nabla v_k|^2dv\to 0.
 $$
 We may assume that
 $v_k$ converges to $v$  strongly in $L^2(M)$ and weakly in $X$ with  $|v|_2=1$. Then $\int_M |\nabla v|^2dv=0$ and $v=\beta$ is a constant. Since $u_k^+$ is locally bounded in $H^1_{loc}(M_-)$, we have $v_+=0$ in $M_-$ and then $\beta<0$ since $|v|_2=1$ and $c_k\to -\infty$. However, this is impossible by \eqref{KW2} and \eqref{KW3}. We may let $u$ be the weak limit of $u_k$ in $H^1(M)$.
 Choosing $\phi=u_k-u$, this then  shows that $(u_k)$ is a convergent PS sequence of $J$.
\end{proof}
  Once we have the compactness result for the functional $J$ as above, the proof of Theorem \ref{thm2} is complete.

\section{Local minimizer for $K\leq 0$}

As we mentioned before, using a very beautiful argument, Kazdan and Warner \cite{KW} have showed that $\alpha_0=\infty$ if and only if $K$ is nontrivial non-positive function on $M$.
 The question now is if one can use the variation method to get a solution in the case treated by Kazdan-Warner. Using the variational method on convex set we answer this question affirmatively.
Precisely, we prove the following result by the variational method on convex set.
 \begin{Thm}\label{mali1} On the any dimensional compact Riemannian manifold $(M,g)$, for $\alpha\in (-\infty,0)$ and for any non-trivial smooth function $K\leq 0$on $M$, there is a solution to \eqref{KW}, which is a local minimizer of the functional defined by
 $$
F(u)=\int_M(|\nabla u|^2+2\alpha u)dv-\int_MK(x) e^{2u}dv, \ u\in H^1(M).
$$
 \end{Thm}
 This result is basically obtained in \cite{KW}, where they have used the monotone method. Here we prefer to give a variational proof for completenesswhich also complements Berger's program of application of variational methods to problems of prescribed non-positive Gauss curvature on closed surface, which are of the same type to \eqref{KW}. The variational method was used by C.Hong \cite{H} to study the problem (\ref{KW}) on the two-sphere.
 \begin{proof}
Recall that $K\leq 0$ on $M$ and $\bar{K}=\frac{1}{|M|}\int_M Kdv<0$.
It is well-known that by the direct method, we can solve the Poisson equation
$$
-\Delta w=K-\bar{K}, \ \ in  \ M
$$
to get the smooth solution $w$ with $w>0$ on $M$. We take $b>0$ and $b=e^r$ and let
$$
v=bw+r.
$$
Choose $b>0$ such that $-b\bar{K}+\alpha>0$.
Note that
$$
-K(e^{bw}-1)\geq 0.
$$
Then we have
\begin{align*}
-\Delta v+\alpha-Ke^v &=-b \Delta w-Ke^{bw+r} +\alpha  \\
  &=-b (\bar{K}-K)-bK e^{bw}+\alpha \\
  &=-b\bar{K}-bK( e^{bw}-1)+\alpha  \\
  &> 0.
\end{align*}
This implies that $u_+:=v$ is a super solution to (\ref{KW}).

Similar to \cite{B}, we
define
$$
F(u)=\int_M(|\nabla u|^2+2\alpha u)dv-\int_MK(x) e^{2u}dv
$$
on the convex set $H:= \{v\in H^1(M), v\leq u_+\}$, which is a convex functional on $H$ by the condition $K\leq 0$ on $M$. Here $v\leq u_+$ means that $v(x)\leq u_+(x)$ almost everywhere on $M$. Set $\mu=\inf_H F(u)$. We shall show that $\mu>-\infty$. Assume that $\mu=\infty$, and we shall have a sequence $\{u_k\}\subset H$ such that
$F(u_k)\to \mu=-\infty$. Note that
$$
F(u)=\int_M|\nabla u|^2dv-\int_MK(x) e^{2u}dv+2\alpha |M| \bar{u}.
$$
Then we must have
$$
\bar{u_k}\to \infty,
$$
which is impossible by the constraint condition
$$
u_k\leq u_+.
$$
Then $\mu>-\infty$ and $\bar{u_k}$ is uniformly bounded. Hence $\{u_k\}\subset H^1(M)$ is a bounded sequence in $H^1(M)$ and then we have
a weakly convergent subsequence, still denoted by $u_k$ with limit $u$. Hence
$
F(u)=\mu
$
and one can directly verify \cite{M2} that $u$ is a solution to \eqref{KW} on $M$ and a local minimizer of the functional $F$ \cite{M2}.  This completes the proof of the result.
\end{proof}


\begin{thebibliography}{20}
\bibitem{A}
T.~Aubin, \emph{Some nonlinear problems in Riemannian geometry}. Springer, 1998.

\bibitem{B}
Melvyn Berger,\emph{ On Riemannian structures of prescribed Gauss curvature for compact two-dimensional manifolds}, J. of Diff. Geo. 5 (1971), 325-332.

\bibitem{St}
F. Borer,  Luca Galimberti, M.Struwe, \emph{"Large'' conformal metrics of prescribed Gauss curvature on surfaces of higher genus}. Comment. Math. Helv.  90  (2015),  no. 2, 407-428.

\bibitem{CY}
S. Y. A. Chang, P. C. Yang, \emph{Prescribing Gaussian curvature on $S^2$}, Acta Math.,  159(1987)215-259.

\bibitem{C}
K.C.Chang, \emph{Infinite dimensional Morse theory and multiple solution problems}. Birkhauser, 1992

\bibitem{CL}
W.Chen, C.Li,  \emph{Gauss curvature in the negative case}. Proc A.M.S., 131(2003), 741-744.

\bibitem{DL}
W.Ding, J.Liu,\emph{ A note on the problem of prescribing Gaussian curvature on surfaces}. Trans. Amer. Math. Soc.  347  (1995),  no. 3, 1059-1066.

\bibitem{KW1} J.L.Kazdan, F. Warner, \emph{Prescribing curvatures}. Proc. Symposia in Math., vol.27, 1975, pp.309-391.

\bibitem{KW}
J.~Kazdan and F.~Warner, \emph{Curvature functions for compact 2-manifolds}. Annals
Math. 99(1974)14-47.

\bibitem{He}
E.Hebey, Nonlinear analysis on manifolds: Sobolev spaces and inequalities, CIMS
Lecture Notes, Courant Institute of Mathematical Sciences, Volume 5, 1999.

\bibitem{H}
Chongwei Hong, \emph{Equation $\Delta u+K(x)e^{2u}=f(x)$ on $R^2$ via stereographic projection}, Journal of mathematical analysis and applications,
130(1988)484-492.

\bibitem{MH}
Li Ma, Min-Chun Hong, \emph{Curvature flow to the Nirenberg problem}. Arch. Math. (Basel) 94 (2010), no. 3, 277-289.

\bibitem{M}
 Li Ma, \emph{Nirenberg's
problem in 90's}, in Differential Geometry of sub-manifolds,
pp.171-177, U.Simon, etc, ed., 2000, World Scientific.


\bibitem{M2}
Li Ma, \emph{Mountain pass on a closed convex set}, J. Math. Anal. and
Applications, 205(1997)531-536.


\bibitem{MW1}
 Li Ma, J.Wei, \emph{Stability of the Lichnerowicz equation}. Journel Math.Pure Appl., 99 (2013) 174-186.

\bibitem{MW}
Li Ma; Ingo Witt, \emph{Discrete Morse flow for Ricci flow and porous medium equation}. Commun. Nonlinear Sci. Numer. Simul. 59 (2018), 158-164. (Reviewer: Kin Ming Hui) 53C44 (35K55)



\end{thebibliography}
\end{document}